\documentclass{amsart}
\usepackage{amsmath}
\usepackage{mathrsfs}
\usepackage[line,color,poly,all]{xy}
\usepackage{graphicx,hyperref,url,xy,amssymb,booktabs, verbatim}
\usepackage[noadjust]{cite}
\usepackage{xspace}
\usepackage[nomarkers, nolists, noheads]{endfloat}

\newcommand{\RR}{\mathbb R}%Reals
\newcommand{\ZZ}{\mathbb Z}%Integers
\newcommand{\QQ}{\mathbb Q}%Rationals
\newcommand{\CC}{\mathbb C}%Complex
\newcommand{\FF}{\mathbb F}%Complex
\newcommand{\z}{\zeta}%
\newcommand{\uu}{\mathbf u}%
\newcommand{\OO}{\mathcal O}%Ring of integers

\newcommand{\cT}{\mathcal{T}}

\newcommand{\cH}{\mathcal{H}}
\newcommand{\cA}{\mathcal{A}}

\newcommand{\cK}{\mathcal{K}}
\newcommand{\cS}{\mathcal{S}}
\newcommand{\vep}{\varepsilon}
\newcommand{\ep}{\epsilon}
\DeclareMathOperator{\Int}{Int}
\newcommand{\fn}{\mathfrak{n}}

\newcommand{\fb}{\mathfrak{b}}
\newcommand{\fa}{\mathfrak{a}}
\newcommand{\fp}{\mathfrak{p}}
\newcommand{\fq}{\mathfrak{q}}
\newcommand{\fr}{\mathfrak{r}}

\newcommand{\group}[1]{\mathbf{#1}}%Algebraic groups

\DeclareMathOperator{\SL}{SL}%Special linear
\DeclareMathOperator{\SO}{SO}%Special Orthogonal
\DeclareMathOperator{\GL}{GL}%General linear
\DeclareMathOperator{\SU}{SU}% special unitary group
\DeclareMathOperator{\Mat}{Mat}% matrices
\DeclareMathOperator{\Sym}{Sym}%Symmetic  

\newcommand{\bG}{\group{G}}

\DeclareMathOperator{\Herm}{Herm}

\DeclareMathOperator{\Eis}{Eis}%Eisenstein

\newcommand{\del}{\partial}%boundary
\DeclareMathOperator{\Tr}{Tr}%trace
\DeclareMathOperator{\Res}{Res}%Restriction of scalars
\DeclareMathOperator{\Norm}{Norm}%Norm

\newcommand{\innprod}[2]{\langle #1, #2 \rangle}%another inner product
\newcommand{\isomto}{\overset{\sim}{\rightarrow}}%"Isomorphic to" arrow
\DeclareMathOperator*{\Inf}{Inf}%Infimum
\DeclareMathOperator{\sgn}{sgn}%Sign
\newcommand{\Vor}{Voronoi\xspace}% Voronoi
\newcommand{\vect}[1]{\begin{bmatrix} #1 \end{bmatrix}}%In order to make vectors
\newcommand{\mat}[1]{\begin{bmatrix} #1 \end{bmatrix}}%
\newcommand{\bu}{\mathbf u}%
\newcommand{\ourprime}{12379}
\newcommand{\cusp}{\operatorname{cusp}}
\newcommand{\fatten}[1]{\phantom{x!}#1\phantom{x!}}
\newcommand{\St}{\mbox{St}}

\newcommand{\longriso}{\xrightarrow{\phantom{\alpha }\sim\phantom{\alpha }}}
\DeclareMathOperator{\Size}{Size}%size
\DeclareMathOperator{\IInt}{{\mathcal{I}}}

\theoremstyle{plain}
\newtheorem{thm}{Theorem}[section]

\newtheorem{prop}[thm]{Proposition}

\theoremstyle{definition}
\newtheorem{defn}[thm]{Definition}
\newtheorem*{remark}{Remark}

\newtheorem{ex}[thm]{Example}

\begin{document}

\title[Modular forms and elliptic curves]{Modular forms and elliptic
curves over the cubic field of discriminant $-23$}

\author{Paul E. Gunnells}

\address{Department of Mathematics and Statistics\\University of
Massachusetts\\Amherst, MA 01003-9305}

\email{gunnells@math.umass.edu}

\author{Dan Yasaki}

\address{Department of Mathematics and Statistics\\ 
University of North Carolina at Greensboro\\Greensboro, NC 27402-6170}

\email{d\_yasaki@uncg.edu}

\renewcommand{\setminus}{\smallsetminus}

\date{\today} 

\thanks{PG was partially supported by NSF grant DMS 1101640.  We thank
Avner Ash, Farshid Hajir, Alan Reid, and Siman Wong for helpful
conversations.  We also thank the referee for many useful suggestions.}



\begin{abstract}
Let $F$ be the cubic field of discriminant $-23$ and let
$\OO\subset F$ be its ring of integers.  By explicitly computing
cohomology of congruence subgroups of $\GL_{2} (\OO)$, we
computationally investigate modularity of elliptic curves over $F$.
\end{abstract}

\maketitle

\section{Introduction}\label{intro}

\subsection{} Let $F$ be a number field with ring of integers $\OO$,
and let $E$ be an elliptic curve defined over $F$.  According to a
general contemporary philosophy with roots in the work of Taniyama,
Shimura, and Weil, the curve $E$ should be \emph{modular}.  There are
many different perspectives about what this might mean, but at the
very least one expects there exists a cuspidal automorphic form $f$ on
$\GL_2$ that is a Hecke eigenform with integral Hecke eigenvalues, and
such that these eigenvalues are closely related to the number of
points on $E$ over various finite fields.  In particular, suppose
$\fn$ is the conductor of $E$.  Then one expects $f$ also to have
conductor $\fn$, and that there should be an identification of partial
$L$-functions $L^{S} (s,f) = L^{S} (s,E)$, where $S$ is a fixed finite
set of primes including all those dividing $\fn$.  This identification
is done through matching of Euler factors at good primes: if $\fp$ is
prime to $\fn$, then one wants
\begin{equation}\label{eq:pointcount}
|E (\FF_{\fp})| = \Norm (\fp) +1 -
a_{\fp},
\end{equation}
where $\FF_{\fp} = \OO/\fp$ is the residue field and $a_{\fp}\in
\ZZ$ is the $\fp$th Hecke eigenvalue.  This idea has been
computationally investigated by a variety of authors, including
Cremona and his students for $F$ complex quadratic \cite{cremona2,
crem.whitley,lingham, bygott}; Socrates--Whitehouse and
Demb{\'e}l{\'e} for $F$ real quadratic \cite{sw, dembele}; and the
current authors, in joint work with F. Hajir, for $F$ the $CM$ field
of fifth roots of unity \cite{zeta5}.

In this paper, we continue this computational work and study the
modularity of elliptic curves when $F$ is the complex cubic field of
discriminant $-23$.  This field has signature $(1,1)$ and is thus not
Galois (its Galois closure is the Hilbert class field of $\QQ
(\sqrt{-23})$).

The overall program is similar to that of \cite{zeta5}.  Namely,
instead of explicitly working with automorphic forms, we work with the
cohomology of the congruence subgroups $\Gamma_{0} (\fn) \subset
\GL_{2} (\OO)$.  By Franke's proof of Borel's conjecture
\cite{Fra}, one knows that this cohomology is built from certain
automorphic forms, and that there is a subspace of \emph{cuspidal
cohomology} that corresponds to certain cuspidal automorphic forms.
This cohomology can be computed using topological tools; in particular
the first step is carrying out a version of explicit reduction theory
due to Koecher \cite{koecher}.  This reduction theory is one
generalization of Voronoi's theory of perfect quadratic forms
\cite{voronoi1} to a much broader setting that includes quadratic
forms over both number fields of arbitrary signature and the
quaternions.  To the best of our knowledge, this is the first time
Koecher's work has been used for such computations.  Since we expect
that these techniques will be useful for later researchers, we take
some time to explain Koecher's work.

The second step is computing the cohomology and its decomposition
under the action of the Hecke operators.  Computing the cohomology is
straightforward, since the reduction theory provides us with an
explicit cell complex with an action by $\Gamma_{0} (\fn)$.  Computing
the Hecke operators, however, is much more involved.  Our technique is
to use the \emph{sharbly complex} and a variant of an algorithm
described by the first-named author for subgroups of $\SL_{4} (\ZZ)$
\cite{Gu}, and later developed by both of us in \cite{ants, zeta5} for
subgroups of $\GL_{2} (\OO)$, where $F$ is real quadratic and totally
complex quartic.  Although these settings seem to have little to do
with each other, they share the feature that the highest cohomological
degree where the cusp forms can contribute is exactly one less than
the virtual cohomological dimension.  Thus although the symmetric
spaces and underlying fields are completely different, computing the
Hecke operators reduces to a similar combinatorial problem.  The
technique we use is closely related to that of modular symbols
\cite{Man, Aru}, although the combinatorics are more complicated since
we need to work with a different cohomology group.  As in the papers
\cite{Gu, ants, zeta5}, we do not have a proof that our technique of
computing Hecke operators works, but the technique has always been
successful in practice.

With the first two steps completed, we find ourselves with a list of
Hecke eigenclasses, some of which apparently correspond to cuspforms,
and further some of which have eigenvalues that are rational integers.
The final step is to compile a list of elliptic curves over $F$ of
small conductors.  This is done in the obvious way: we enumerate
curves by searching over a box of Weierstrass equations and throwing
out duplicates.

After all three steps are done, we arrive at a list of cohomology
classes and a list of elliptic curves.  We found complete agreement
between these lists.  In particular:

\begin{itemize}
\item For each elliptic curve $E$ with norm conductor within the range
of our cohomology computations, we found a cuspidal cohomology class
with rational Hecke eigenvalues that matched the point counts for $E$
as in \eqref{eq:pointcount}, for every Hecke operator that we checked.
\item For each cuspidal cohomology class with rational Hecke
eigenvalues, we found a corresponding elliptic curve whose point
counts matched every eigenvalue we computed.
\end{itemize}

We now give a guide to the contents of the paper.  In \S\S
\ref{s:pd},\ref{s:qfnf} we recall Koecher's work on positivity
domains, and explain how one can find explicit fundamental domains for
$\GL_{n}$ over number fields acting on spaces of quadratic forms.  In
\S \ref{s:kp} we compute the Koecher polyhedron for the case under
consideration, namely $\GL_{2} (\OO)$ where $\OO$ is the ring of
integers in the cubic field of discriminant $-23$. Next, \S
\ref{s:cohomology} explains the complexes we use to compute the
cohomology of subgroups of $\GL_{2} (\OO)$ and the Hecke action.
After this, in \S \ref{sec:reduction} we give a brief overview of how
we compute the action of the Hecke operators, following \cite{zeta5,
ants}.  The following two sections \S\S
\ref{s:rl0},\ref{sec:normal_forms} give more details about what is
done differently in the current paper from the techniques in
\cite{ants, zeta5}.  Finally, in \S \ref{s:cr} we present our
computational data.

\section{Positivity domains}\label{s:pd}

In this section we review Koecher's study of positivity domains
\cite{koecher}, which generalizes Voronoi's reduction theory for
positive definite quadratic forms \cite{voronoi1}.  

Let $V$ be a finite dimensional vector space over $\RR$.  Let $\innprod{\phantom{a}}{\phantom{a}}
\colon V \times V \rightarrow \RR$ be a positive definite symmetric
bilinear form.  The form determines a norm $|\cdot|$ on $V$
in the usual way by $|v| = \sqrt{\innprod{v}{v}}$, and we give $V$ the
metric topology.  For any subset $C\subset V$, let $\overline C$ be
the closure, let $\Int (C)$ be the relative interior, and let
$\partial C = \overline C \smallsetminus \Int (C)$ be the boundary.

\begin{defn}\label{def:pd}
A subset $C \subset V$ is called a
\emph{positivity domain} if the following hold:
\begin{itemize}
\item $C$ is open and nonempty.
\item $\innprod{x}{y}>0$ for all $x,y\in C$.
\item For each $x\in V \smallsetminus C$ there is a nonzero $y\in
\overline C$ such that $\innprod{x}{y}\leq 0$. 
\end{itemize}
\end{defn}

One can show that a positivity domain $C$ is a convex cone in $V$.  In
other words, if $y\in C$ then $\lambda y\in C$ for all $\lambda > 0$,
and if $y,y'\in C$ then $y+y'\in C$.  Furthermore, $C$ contains no
line. 

Let $D\subset \overline C \smallsetminus \{0 \}$ be a nonempty
discrete subset.  For $y\in C$, let 
\[
\mu (y) = \Inf_{d\in D} \{\innprod{d}{y} \}.
\] 
Koecher proves that $\mu (y)>0$ and that the infimum is achieved only
 on a finite set of points, which we denote $M (y)$:
\[
M (y) = \{d\in D\mid \innprod{d}{y}=\mu (y) \}.
\]
We call $M (y)$ the set of \emph{minimal vectors} for $y$. 

\begin{defn}\label{def:perfect}
A point $y\in C$ is called \emph{perfect} if the linear span of its
minimal vectors $M (y)$ is all of $V$.
\end{defn}

Note that the set $M (y)$, as well as the notion of perfection of a
point in $C$, depend on the choice of the set $D$.
If $y$ is perfect, then so is $\lambda y$ for $\lambda > 0$, and
clearly $M (\lambda y) = M (y)$.  We let $\Phi = \Phi (D)$ be the set
of all perfect points $y$ with $\mu (y)=1$.

\begin{ex}\label{ex:classical}
We consider the classical case of this construction.  Let $V$ be the
space of symmetric $n\times n$ matrices over $\RR$, let
$\innprod{a}{b} = \Tr (ab)$, and let $C\subset V$ be the cone of
positive definite matrices.  Let $D\subset V$ be the set of points of
the form $vv^{t}$ for $v\in \ZZ^{n}$, where we regard elements of
$\ZZ^{n}$ as column vectors.  Then if $y\in C$, $d = vv^{t}\in D$, the
quantity $\innprod{d}{y}$ is easily seen to be the value $Q_{y} (v) =
v^{t}yv$, in other words the value of the positive definite quadratic
form $Q_{y}$ determined by $y$ on the integral vector $v$.  Koecher's
notion of perfect coincides with Voronoi's notion of a perfect
quadratic form: a positive definite quadratic form is perfect if it
can be recovered from the knowledge of its nonzero minimum and the set
of vectors on which the minimum is attained.
\end{ex}

Returning now to the general setting, 
Voronoi used perfect quadratic forms to provide an explicit reduction
theory for the cone of positive definite quadratic forms, and we want
the same for our positivity domain $C$.  Unfortunately not all sets
$D$ will accomplish this.  Thus we have the following definition:

\begin{defn}\label{def:admissible}
A nonempty discrete subset $D\subset \overline{C}\smallsetminus \{0
\}$ is said to be \emph{admissible} if for any sequence $\{y_{i} \}$
converging to a point in $\partial C$, we have $\lim \mu (y_{i}) = 0$.
\end{defn}

Note that the set $D$ from Example \ref{ex:classical} is admissible.
Suppose $\{y_{i} \}\rightarrow y\in \partial C$.  The real quadratic
form $Q_{y}$ associated to $y$ is degenerate, and thus there is a
nonzero subspace $W\subset \RR^{n}$ such that $Q_{y}$ restricted to
$W$ vanishes.  Since there exist integral vectors arbitrarily close to
$W$, we have $\lim \mu (y_{i}) = 0$.

Koecher proved that if $D$ is admissible, then $\Phi (D)$ is a
discrete subset of $C$, and provides a polyhedral decomposition of
$C$.  To explain what we mean, we must recall some notions from convex
geometry (a convenient reference is \cite[Chapter 1]{fulton}). Recall
that a \emph{polyhedral cone} in a real vector space $V$ is a subset
$\sigma$ of the form
\[
\sigma = \sigma (v_{1},\dotsc ,v_{p}) = \{\sum_{i=1}^{p}
\lambda_{i}v_{i}\mid \lambda_{i}\geq 0 \},
\]
where $v_{1},\dotsc ,v_{p}$ is a fixed set of vectors.  We say that
the $v_{1},\dotsc ,v_{p}$ \emph{span} $\sigma$.  The dimension of $\sigma$ is
the dimension of its linear span; if the dimension of $\sigma$ is $n$,
then we call $\sigma$ an \emph{n-cone}. If $\sigma$ is spanned by a
linearly independent subset, then $\sigma$ is called
\emph{simplicial}.  

Now given any $y\in \Phi (D)$, let $\sigma (y)$ be the cone
\[
\sigma (y) = \{\sum \lambda_{d}d\mid \lambda_{d}\geq 0, d\in M (y) \}.
\]
Koecher calls $\sigma (y)$ the \emph{perfect pyramid} of $y$.  We
remark that typically $\sigma (y)$ is not a simplicial cone.  

Let $\Sigma$ be the set of perfect pyramids and all their proper
faces, as $y$ ranges over all points in $\Phi (D)$.  Koecher proves
that for admissible $D$, the perfect pyramids have the following
properties:
\begin{enumerate}
\item Any compact subset of $C$ meets only finitely many perfect
pyramids.  
\item Two different perfect pyramids have no interior point in
common.\label{interior}
\item Given any perfect pyramid $\sigma$, there are only finitely many
perfect pyramids $\sigma '$ such that $\sigma \cap \sigma '$ contains
a point of $C$ (which, by item \eqref{interior}, must lie on the
boundaries of $\sigma$, $\sigma '$).\label{neighbors}
\item The intersection of any two perfect pyramids is a common face of
each.  \label{fan}
\item Let $\sigma (y)$ be a perfect pyramid and $F$ a codimension one
face of $\sigma (y)$.  If $F$ meets $C$, then there is another perfect
pyramid $\sigma (y')$ such that $\sigma (y)\cap \sigma (y') =
F$.\label{facetneighbor}  
\item We have $\bigcup_{\sigma \in \Sigma} \sigma \cap C = C$. \label{complete}
\end{enumerate}
In item \eqref{facetneighbor}, if a facet $F$ of a perfect pyramid is
contained in $\partial C$, we say that $F$ is a \emph{dead end}.  In
item \eqref{neighbors}, the pyramid $\sigma '$ is called a
\emph{neighbor} of $\sigma$.  Property \eqref{fan}, together with the
definition of $\Sigma$, implies that $\Sigma$ is a \emph{fan}
\cite[\S1.4]{fulton}.

\begin{defn}\label{def:koecherfan}
We call $\Sigma$ the \emph{Koecher fan}, and the cones in $\Sigma$
the \emph{Koecher cones}.
\end{defn}

Now we bring groups into the picture.  Let $G\subset \GL (V)$ be the
group of automorphisms of $C$.  Let $\Gamma \subset G$ be a discrete
subgroup such that $\Gamma D=D$.  Koecher proves that if $D$ is
admissible, then $\Gamma$ acts properly discontinuously on $C$
\cite[\S5.4]{koecher}.  Moreover, the Koecher fan gives an explicit
reduction theory for $\Gamma$ in the following sense:

\begin{enumerate}
\item [\textbf{(O1)}] There are finitely many $\Gamma$-orbits in $\Sigma $.
\item [\textbf{(O2)}] Every $y\in C$ is contained in a unique cone in $\Sigma$. 
\item [\textbf{(O3)}] Given any cone $\sigma \in \Sigma$ with $\sigma \cap C \not
=\emptyset$, the group $\{ \gamma \in \Gamma \mid \gamma \sigma =
\sigma \}$ is finite.
\end{enumerate}

Choose representatives  $\sigma_{1},\dotsc ,\sigma_{n}$ of the orbits
of $\Gamma$ in $\Sigma $, and let 
\[
\Omega = \bigcup \sigma_{i}\cap C.
\]
One would hope that $\Omega$ is a fundamental domain for the action of
$\Gamma $ on $C$, but unfortunately the properties
\textbf{(O1)}--\textbf{(O3)} do not imply this.  However $\Omega$ is
close to a fundamental domain: each of the $\sigma_i$ has at worst a
finite stabilizer subgroup in $\Gamma$.  If one wishes to refine
$\Omega$ to an exact reduction domain, as Koecher does, one must
account for these finite stabilizers, for instance by passing to the
barycentric subdivision of the perfect pyramids and then taking an
appropriate union of the resulting cones.  For our purposes this is
not necessary, since explicit knowledge of $\Omega$ and these
stabilizers suffices for cohomology computations
(cf.~\cite[\S4]{AGM}).

We conclude this section by presenting another viewpoint on the
perfect pyramids and the Koecher fan.  This perspective does not
appear in \cite{koecher}, but instead is motivated by the theory of
cores and co-cores in \cite[Chapter II]{amrt}.  

\begin{defn}\label{def:koecherpolyhedron}
Let $D$ be an admissible set in $\overline{C}$.  The \emph{Koecher
polyhedron} $\Pi$ is the convex hull in $\overline{C}$ of $D$.
\end{defn}

The connection between $\Pi$ and $\Sigma$ is given by the following proposition:

\begin{prop}\label{prop:sigmavspi}
Assume that no perfect pyramid in $\Sigma$ has a dead end.  Then the
nonzero cones in the fan $\Sigma$ are the cones on the faces of $\Pi$.
\end{prop}

\begin{proof}
Because $\Pi$ is a convex polyhedron, it suffices to see that the
perfect pyramids are the cones on the facets (maximal proper faces) of
$\Pi$.  So let $y$ be a perfect form with minimal vectors
$d_{1},\dotsc ,d_{n}$ and with perfect pyramid $\sigma (y)$.  The
equation $\innprod{x}{y}=1$ defines a hyperplane in $V$ that is a
supporting hyperplane for $\Pi$ (all other $d\in D$ satisfy
$\innprod{d}{y}>1$, and $\Pi$ is the convex hull of $D$).  Since the
$d_{i}$ span $V$, it follows that the convex hull of the $d_{i}$ is a
facet of $\Pi$.  Thus $\sigma (y)$ gives rise to a facet $F$ of $\Pi$,
and $\sigma (y)$ is the cone on $F$.  Since we assume no perfect
pyramids have dead ends, property \eqref{facetneighbor} on page
\pageref{facetneighbor} implies that the facets of $\Pi$ meeting $F$
in a codimension one face correspond to perfect forms by a similar
argument.  Thus the facets of $\Pi$ are in bijection with the perfect
pyramids.
\end{proof}

\section{Quadratic forms over number fields and reduction theory}\label{s:qfnf}

In this section we specialize the results of \S\ref{s:pd} to the
setting of primary concern to us.  This recapitulates \cite[\S
9]{koecher}.

Let $F$ be a number field of degree $d = r + 2s$ with $r$ real
embeddings and $s$ conjugate pairs of complex embeddings.  For each
pair of complex conjugate embeddings, choose and fix one.  We can then
identify the infinite places of $F$ with its real embeddings and our
choice of complex embeddings.

For each infinite place $v$ of $F$, let $V_{v}$ be the real vector
space of $n\times n$ real symmetric (respectively, of complex
Hermitian) matrices $\Sym_{n} (\RR)$ (resp., $\Herm_{n} (\CC)$) if $v$
is real (resp., complex).  Let $C_{v}$ be the corresponding cone of
positive definite (resp., positive Hermitian) forms.  Put $V =
\prod_{v}V_{v}$ and $C = \prod_{v}C_{v}$, where the products are taken
over the infinite places of $F$.  We equip $V$ with the inner product
\begin{equation}\label{eq:cv}
\langle x,y\rangle = \sum_{v} c_{v}\Tr (x_{v}y_{v}),
\end{equation}
where the sum is again taken over the infinite places of $F$, and
$c_v$ equals $1$ if $v$ is real and equals $2$ if $v$ is complex.

The group $G = \GL_{n} (\RR)^{r} \times \GL_{n} (\CC)^{s}$ acts on $V$
by
\[
(g\cdot y)_{v} = \begin{cases}
g_{v}y_{v}g_{v}^{t}&\text{$v$ real,}\\
g_{v}y_{v}\bar g_{v}^{t}&\text{$v$ complex.}\\
\end{cases}
\]
This action preserves $C$, and exhibits $G$ as the full automorphism
group of $C$.  Moreover, if $\bG$ is the reductive group $\Res_{F/\QQ}
\GL_{n}$, then we can identify the group of real points $\bG (\RR)$
with $G$.  In fact, we can identify the quotient $C/\RR_{>0}$ of $C$
by homotheties with the symmetric space $X = G/KA_{G}$, where
$K\subset G$ is the maximal compact subgroup and $A_{G}$ is the split
component (the connected component of a maximal $\QQ$-split torus in
the center of $\bG$).  Indeed, let $X_{n}$ (respectively $Y_{n}$) be
the symmetric space of noncompact type $\SL_{n} (\RR)/\SO (n)$
(resp. $\SL_{n} (\CC)/\SU (n)$).  Then we have 
\[
X \simeq (X_{n})^{r}\times (Y_{n})^{s}\times Z, 
\]
where $Z$ is a euclidean symmetric space of dimension
$r+s-1$.\footnote{In terms of the associated discrete groups, this
flat factor $Z$ accounts for the difference between $\SL_{n} (\OO)$
and $\GL_{n} (\OO)$; the former is infinite index in the latter.}  Now
we also have isomorphisms
\begin{equation}\label{eq:isos}
\Sym_{n} (\RR)/\RR_{>0} \isomto X_{n}, \quad \Herm_{n}
(\CC)/\RR_{>0} \isomto Y_{n}.
\end{equation}
The identification $C/\RR_{>0} \simeq X$ then follows, once one
rewrites \eqref{eq:isos} as $\Sym_{n} (\RR) \simeq X_{n}\times
\RR_{>0}$ and $\Herm_{n} (\CC) \simeq Y_{n}\times \RR_{>0}$.

Now we construct an admissible subset $D$ of $\overline {C}$.  Let
$\OO$ be the ring of integers of $F$.  The nonzero (column) vectors
$\OO^{n}\smallsetminus \{0 \}$ determine points in $V$ via
\begin{equation}\label{eq:star}
q \colon x \longmapsto (x_{v}x_{v}^{*}).
\end{equation}
Here $*$ denotes transpose if $v$ is real, and conjugate transpose if
$v $ is complex.  We have $q (x)\in \overline{C}$ for all $x\in
\OO^{n}$, and by analogy with Voronoi's construction it is natural to
consider the set $D$ obtained by applying $q$ to the nonzero points.
Indeed, we have the following basic result of Koecher
(cf.~\cite[Lemma 11]{koecher}):

\begin{prop}\label{prop:admissible}
The set 
\[
D = \bigl\{ q (x) \bigm| x\in \OO^{n}\smallsetminus \{0 \}\bigr\}
\]
is admissible.
\end{prop}

The group $\Gamma =\GL_{n} (\OO)$ acts on $C$ and takes $D$ into
itself.  Thus we can find the Koecher fan $\Sigma $.  After passing to
the symmetric space $X$, we obtain a decomposition of $X$ into cells
with $\Gamma$-action that can then be used to compute the cohomology
of $\Gamma $ and its finite-index subgroups as in \cite{AGM, zeta5}
and other places.

\begin{remark}
One can think of the cone $C$ as being the space of real-valued
positive quadratic forms over $F$ in $n$-variables.  Specifically, if
$A\in C$ is a tuple $(A_{v})$, then $A$ determines a quadratic form
$Q_{A}$ on
$F^{n}$ by 
\[
Q_{A} (x) = \sum c_v x_v^* A_v x_v,
\]
where $c_{v}$ is defined in \eqref{eq:cv} and $*$ is defined as in
\eqref{eq:star}.  However, not every quadratic form of interest comes
from a matrix $A$ that is the image of a matrix from $F$ under the
embeddings.  For instance, in general the perfect forms will not come
from matrices over $F$ in this way.
\end{remark}

\section{The Koecher polyhedron and fan for the field of
discriminant $-23$}\label{s:kp} 

From now on, we set $n = 2$ and consider a particular mixed signature
cubic field, namely $F = \QQ [x] / (f(x))$, where $f (x)= x^3 - x^2 +
1$.  If we choose a root $t$ of $f$, then the field $F$ has ring of
integers $\OO = \ZZ[t]$, discriminant $-23$, and class number one.  We
remark that much of what we do can be extended to other mixed
signature cubic fields, but for this investigation we focus on $F$ to
honor its special place in the menagerie of complex
cubics.\footnote{For instance, it appears first in lists of cubic
fields ordered by the absolute value of their discriminants.  It
appears often in algebraic number theory courses since its Galois
closure is the Hilbert class field of $\QQ (\sqrt{-23})$.}

Since $F$ has signature $(1,1)$, the space of forms $V$ is the
product
\[V = \Sym_2(\RR) \times \Herm_2(\CC).\]  The quotient $X =
C/\RR_{>0}$ is $6$-dimensional.

We now present representatives for the $\Gamma = \GL_2(\OO)$ orbits of
perfect forms.  These were computed using the algorithm in
\cite{Gmod}, which is a recasting of \Vor's original algorithm to the
setting of self-adjoint homogeneous cones.  In particular, the
algorithm begins with an initial perfect form and then identifies the
neighbors of the corresponding perfect pyramid.  To find the initial
perfect form, we computed a large enough part of the Koecher
polyhedron $\Pi$ (Definition \ref{def:koecherpolyhedron}) to identify
a facet.  

In our computation, we found that there are nine $\Gamma$-orbits of
perfect forms, which thus give rise to nine $\Gamma$-orbits of perfect
pyramids.  Seven of these orbits consist of simplicial cones (each with
seven spanning vectors); the remaining two classes of cones are
spanned by eight and nine vectors.

\begin{thm}\label{thm:perfect}
A perfect binary form over $F$ has minimal vectors that are
$\GL_2(\OO)$-conjugate to exactly one of the following sets (each
vector $d$ in these lists is a representative of a pair $\pm d$ of
minimal vectors):
\begin{gather*}
\left\{ \vect{1\\0}, \vect{0\\1}, \vect{1\\1}, \vect{t^2 - t\\t^2},
\vect{-t\\-t},     \vect{1\\-t^2 + 1}, \vect{0\\-t} \right\}, \\
    \left\{ \vect{1\\0}, \vect{0\\1}, \vect{-t\\-t}, \vect{1\\-t^2 +
      1}, \vect{t^2 - t\\0},     \vect{0\\-t}, \vect{t^2 - t\\t^2}
    \right\},\\ 
    \left\{ \vect{1\\0}, \vect{0\\1}, \vect{1\\1}, \vect{t^2\\1},
    \vect{-t^2 + 1\\-t^2},     \vect{-t\\0}, \vect{-t\\-t},
    \vect{1\\-t^2 + 1}, \vect{0\\-t} \right\}, \\
    \left\{ \vect{1\\0}, \vect{0\\1}, \vect{1\\1}, \vect{t^2\\1},
    \vect{t^2 - t\\t^2},     \vect{1\\-t^2 + 1}, \vect{0\\-t}
    \right\},  \\
    \left\{ \vect{1\\0}, \vect{0\\1}, \vect{1\\1},
    \vect{-t\\-t}, \vect{t^2 - t\\0}, \vect{t^2 -    t\\t^2},
    \vect{0\\-t} \right\},\\ 
    \left\{ \vect{1\\0}, \vect{0\\1}, \vect{0\\-t}, \vect{t^2\\t^2},
    \vect{-t\\-t}, \vect{t^2 -     t\\0}, \vect{-1\\-t}, \vect{t\\t^2}
    \right\},\\ 
    \left\{ \vect{1\\0}, \vect{0\\1}, \vect{1\\1}, \vect{t^2 -
      t\\t^2}, \vect{-t\\-t},     \vect{t^2\\t}, \vect{1\\-t^2 + 1}
    \right\},\\ 
    \left\{ \vect{1\\0}, \vect{0\\1}, \vect{t^2\\t^2 - t},
    \vect{1\\t^2 - t}, \vect{-t\\0},     \vect{1\\t^2}, \vect{t^2 -
      t\\-t} \right\},\\ 
    \left\{ \vect{1\\0}, \vect{0\\1}, \vect{1\\1}, \vect{-t\\0},
    \vect{t^2 - t\\-t},     \vect{-t^2\\-t^2 + t}, \vect{t^2 - t\\t^2
      - t} \right\} .
\end{gather*}
\end{thm}

Using the data in Theorem \ref{thm:perfect}, we can compute the
complete combinatorial structure of the fan $\Sigma$.  In particular
we can enumerate the $\GL_{2} (\OO)$-orbits of the lower-dimensional
cones of $\Sigma$.  The results are given in Table~\ref{tab:orbits}.
For later purposes, we record the following facts about $\Sigma$,
which represents the information we need for our purposes.
Propositions \ref{prop:cones} and \ref{prop:stabs} are proved by
direct computation. 

\begin{prop}\label{prop:cones} \mbox{}
\begin{enumerate}
\item All cones up to dimension $6$ in the Koecher fan are simplicial.
\item All $1$-cones lie in the boundary $\partial C$.
\item \label{2cones} One of the two orbits of $2$-cones lies in
$\partial C$; a representative has spanning vectors
\[
\left\{\vect{1\\0}, \vect{t\\0} \right\}.
\]
The other orbit, with representative spanned by 
\begin{equation}\label{eqn:interior2cone}
\left\{\vect{1\\0}, \vect{0\\1} \right\}.
\end{equation}
meets $C$.
\item There are ten orbits of $3$-cones, represented by  
\begin{gather*}
    \left\{ \vect{1 \\0}
    , \vect{0 \\1}
    , \vect{0 \\-t}
     \right\},
    \left\{ \vect{1 \\0}
    , \vect{0 \\1}
    , \vect{1 \\-t^2 + 1}
     \right\},
    \left\{ \vect{1 \\0}
    , \vect{0 \\1}
    , \vect{t - 1 \\-t + 1}
     \right\},\\
    \left\{ \vect{1 \\0}
    , \vect{0 \\1}
    , \vect{-t \\-t}
     \right\},
    \left\{ \vect{1 \\0}
    , \vect{0 \\1}
    , \vect{t \\1}
     \right\},
    \left\{ \vect{1 \\0}
    , \vect{0 \\1}
    , \vect{t^2 - t + 1 \\t - 1}
     \right\},\\
    \left\{ \vect{1 \\0}
    , \vect{0 \\1}
    , \vect{-t + 1 \\-t^2 + t}
     \right\},
    \left\{ \vect{1 \\0}
    , \vect{0 \\1}
    , \vect{t^2 \\t^2}
     \right\},
    \left\{ \vect{1 \\0}
    , \vect{0 \\1}
    , \vect{-t^2 + t \\t^2 - t + 1}
     \right\},\\
    \left\{ \vect{1 \\0}
    , \vect{0 \\1}
    , \vect{1 \\1}
     \right\}.
\end{gather*}
All $3$-cones meet the interior of $C$.
\item There are thirty-one orbits of $4$-cones.  Representatives have
spanning vectors given by the following list (each pair in this list
should be supplemented by the spanning vectors of the cone \eqref{eqn:interior2cone}).
\begin{gather*}
\left\{\vect{1\\-t^2 + 1},\vect{0\\-t}\right\},
\left\{\vect{-t\\-t},\vect{0\\-t}\right\},
\left\{\vect{-t + 1\\-t^2 + t},\vect{t^2 - t + 1\\t - 1}\right\},\\
\left\{\vect{-t\\-t},\vect{1\\-t^2 + 1}\right\},
\left\{\vect{t^2\\t^2},\vect{t\\1}\right\},
\left\{\vect{-t\\0},\vect{0\\t^2 - t}\right\},
\left\{\vect{t^2 - t\\t^2},\vect{1\\-t^2 + 1}\right\},\\
\left\{\vect{0\\-t},\vect{t - 1\\t^2 - t}\right\},
\left\{\vect{-1\\-t + 1},\vect{-t^2 + t\\t^2 - t + 1}\right\},
\left\{\vect{t^2 - t\\t^2},\vect{-t\\-t}\right\},\\
\left\{\vect{-1\\t^2 - t},\vect{t - 1\\t^2 - t}\right\},
\left\{\vect{-t\\0},\vect{-t^2 + t\\t^2 - t + 1}\right\},
\left\{\vect{-t\\0},\vect{-1\\-t + 1}\right\},\\
\left\{\vect{-1\\t^2 - t},\vect{0\\-t}\right\},
\left\{\vect{1\\1},\vect{0\\-t}\right\},
\left\{\vect{t^2 - t\\-t^2 + t},\vect{-t + 1\\t - 1}\right\},
\left\{\vect{1\\1},\vect{1\\-t^2 + 1}\right\},\\
\left\{\vect{t^2\\-t^2 + 1},\vect{t\\-t}\right\},
\left\{\vect{-t\\0},\vect{-t + 1\\t - 1}\right\},
\left\{\vect{-t\\0},\vect{t^2 - t\\-t^2 + t}\right\},\\
\left\{\vect{t^2 - t\\-1},\vect{-t + 1\\-t^2 + t}\right\},
\left\{\vect{1\\1},\vect{t^2 - t\\t^2}\right\},
\left\{\vect{-t\\t^2},\vect{t\\-t}\right\},
\left\{\vect{-t\\t^2},\vect{t^2\\-t^2 + 1}\right\},\\
\left\{\vect{-t^2 + t\\t^2 - t + 1},\vect{-t + 1\\-t^2 + t}\right\},
\left\{\vect{-t^2\\t^2 - t},\vect{-t^2 + t\\-t + 1}\right\},
\left\{\vect{-t^2 + 1\\1},\vect{-t\\-t^2 + 1}\right\},\\
\left\{\vect{1\\t^2},\vect{-t^2 + 1\\1}\right\},
\left\{\vect{t^2\\1},\vect{t^2 - t\\t^2}\right\},
\left\{\vect{-t\\0},\vect{t\\1}\right\},
\left\{\vect{t^2 - t\\-t^2 + t},\vect{t^2\\-t^2 + t}\right\}.
\end{gather*}
\end{enumerate}
\end{prop}

\begin{prop}\mbox{}\label{prop:stabs}
Let $\gamma$ and $\delta$ denote the matrices 
\[
\gamma =\mat{ 0 & 1 \\ 1 & 0 }, \quad \delta =\mat{ 0 & -1 \\ 1 & 0 }.
\]
\begin{enumerate}
\item %The stabilizer of the $2$-cone that lies in $\partial C$ is
  %infinite.  
The $2$-cone that meets $C$ has stabilizer of order $8$,
  isomorphic to $D_4$ generated by $\gamma$ and $\delta$.

\item The stabilizers of the $3$-cones are given in Table \ref{tab:3conestabs}.
\begin{table}
\begin{center}
\begin{tabular}{|c|c|c|}
\hline
Order  & Group & Generators\\ \hline
$4$ & $\ZZ/2\ZZ \times \ZZ/2\ZZ $ & $-I,
\gamma\delta
$\\ 
$4$ & $\ZZ/2\ZZ \times \ZZ/2\ZZ $ & $-I,
\mat{ 1 & -t^2 + 1 \\ 0 & -1 }$\\ 
$4$ & $\ZZ/2\ZZ \times \ZZ/2\ZZ $ & $-I,
\gamma$\\ 
$4$ & $\ZZ/2\ZZ \times \ZZ/2\ZZ $ & $-I,
\gamma$\\ 
$4$ & $\ZZ/2\ZZ \times \ZZ/2\ZZ $ & $-I,
\mat{ -1 & 0 \\ t & 1 }$\\ 
$2$ & $\ZZ/2\ZZ $ & $-I $\\ 
$2$ & $\ZZ/2\ZZ $ & $-I $\\ 
$4$ & $\ZZ/2\ZZ \times \ZZ/2\ZZ $ & $-I,
\gamma$\\  
$2$ & $\ZZ/2\ZZ $ & $-I$\\  
$12$ & $D_6 $ & $\mat{ 1 & 1 \\ -1 & 0 }, \gamma $\\ \hline
\end{tabular}
\end{center}
\medskip
\caption{The stabilizer groups of the $3$-cones in the Koecher fan
$\Sigma$.\label{tab:3conestabs} See Proposition \ref{prop:stabs} for notation.}
\end{table}

\item The stabilizers of the $4$-cones are given in Table \ref{tab:4conestabs}.
\begin{table}
\begin{center}
\begin{tabular}{|c|c|c|}
\hline
Order  & Group & Generators\\ 
\hline
$4$ & $\ZZ/2\ZZ \times \ZZ/2\ZZ$ & $-I, \mat{ 1 & -t^2 + 1 \\ 0 & -1 }$\\ 
$2$ & $\ZZ/2\ZZ$ & $-I $\\
$2$ & $\ZZ/2\ZZ$ & $-I $\\
$2$ & $\ZZ/2\ZZ$ & $-I $\\
$2$ & $\ZZ/2\ZZ$ & $-I $\\
$8$ & $D_4$ & $\gamma\delta, \mat{ 0 & -t^2 + t \\ -t & 0 }$\\
$2$ & $\ZZ/2\ZZ$ & $-I $\\
$2$ & $\ZZ/2\ZZ$ & $-I $\\
$2$ & $\ZZ/2\ZZ$ & $-I $\\
$2$ & $\ZZ/2\ZZ$ & $-I $\\
$2$ & $\ZZ/2\ZZ$ & $-I $\\ 
$2$ & $\ZZ/2\ZZ$ & $-I $\\ 
$2$ & $\ZZ/2\ZZ$ & $-I $\\ 
$4$ & $\ZZ/2\ZZ \times \ZZ/2\ZZ$ & $-I, \mat{ -1 & t^2 - t \\ 0 & 1 }$\\
$4$ & $\ZZ/2\ZZ \times \ZZ/2\ZZ$ & $-I, \mat{ 1 & 1 \\ 0 & -1 }$\\
$4$ & $\ZZ/2\ZZ \times \ZZ/2\ZZ$ & $-I, \gamma$\\ 
$2$ & $\ZZ/2\ZZ$ & $-I $\\ 
$2$ & $\ZZ/2\ZZ$ & $-I $\\ 
$2$ & $\ZZ/2\ZZ$ & $-I $\\ 
$2$ & $\ZZ/2\ZZ$ & $-I $\\ 
$2$ & $\ZZ/2\ZZ$ & $-I $\\ 
$2$ & $\ZZ/2\ZZ$ & $-I $\\ 
$4$ & $\ZZ/2\ZZ \times \ZZ/2\ZZ$ & $-I, \mat{ -t & t^2 \\ -t & t }$\\ 
$2$ & $\ZZ/2\ZZ$ & $-I $\\ 
$2$ & $\ZZ/2\ZZ$ & $-I $\\ 
$2$ & $\ZZ/2\ZZ$ & $-I $\\ 
$2$ & $\ZZ/2\ZZ$ & $-I $\\ 
$2$ & $\ZZ/2\ZZ$ & $-I $\\ 
$4$ & $\ZZ/4\ZZ$ & $\mat{ t^2 & 1 \\ -t^2 + t & -t^2 }$\\
$4$ & $\ZZ/2\ZZ \times \ZZ/2\ZZ$ & $-I, \mat{ -1 & 0 \\ t & 1 }$\\ 
$4$ & $\ZZ/2\ZZ \times \ZZ/2\ZZ$ & $-I, \mat{ t^2 - t & -t^2 + t \\ t^2 & -t^2 + t }$\\ \hline
\end{tabular}
\end{center}
\medskip
\caption{The stabilizer groups of the $4$-cones in the Koecher fan
$\Sigma$.\label{tab:4conestabs} See Proposition \ref{prop:stabs} for notation.}
\end{table}
\end{enumerate}
\end{prop}

\begin{table}
\begin{center}
\begin{tabular}{|c|c|c|c|c|c|c|c|}
\hline
$k$&\fatten1 & \fatten2 & \fatten3& \fatten4& \fatten5 &\fatten6 & \fatten{7}\\
\hline
$N_{k}$ &1& 2 & 10 & 31 & 47 & 35 & 9\\\hline
\end{tabular}
\medskip
\caption{The number $N_{k}$ of $\GL_2(\OO)$-orbits of cones of
  dimension $k$ in the Koecher fan $\Sigma$.\label{tab:orbits}}
\end{center}
\end{table}

\begin{remark}
A consequence of the computations in this section is that the Koecher
fan has no dead ends, in the sense of item \eqref{facetneighbor} on
page \pageref{facetneighbor}.  Thus the cones in the Koecher fan are
in bijection with the cones on the faces of the Koecher polyhedron, as
in Proposition \ref{prop:sigmavspi}.  We expect this is always true
for $\GL_{n}$ over number fields when $D$ is constructed as in
Proposition \ref{prop:admissible}.  We do not know if this is true for
general admissible sets $D$, although we expect not, since dead ends
do occur in other generalizations of Voronoi's theory \cite[\S3]{dsv}.
\end{remark}

\section{Cohomology, the Koecher complex, and the sharbly
complex}\label{s:cohomology}

We now turn to the cohomology spaces $H^{*} (\Gamma ; \CC)$, where
$\Gamma \subset \GL_{2} (\OO)$ is now a congruence subgroup.  This is
isomorphic to $H^{*} (\Gamma \backslash X; \CC)$.  Since $X$ is
$6$-dimensional, we have $H^{i} (\Gamma ; \CC) = 0$ unless $0\leq
i\leq 6$.  Since $\bG $ mod its radical has $\QQ$-rank 1, by the
Borel--Serre vanishing theorem \cite{BS} we have $H^{i} (\Gamma ; \CC)
= 0$ if $i=6$.  Thus the cohomological dimension $\nu$ of any $\Gamma$
is $5$.  One can show that the \emph{cuspidal cohomology} (the
part of the cohomology corresponding to cuspidal automorphic forms)
only occurs in degrees $2,3,4$, and general results imply that we only
need to compute one of these groups.  Thus we focus on $H^{4} (\Gamma
; \CC) $; as we shall see this is the easiest group for us to compute.

We now describe the techniques we use to compute cohomology and the
action of the Hecke operators.  Similar techniques were used in
\cite{zeta5, AGM}.  For more details about the complexes we use (in
the setting of $\GL_{n}/\QQ $), see \cite{AGM5}.

Let $\cT$ be the Tits building for $\GL_{n}/F$.  Thus $\cT$ is a
simplicial complex with $k$-simplices given by the proper flags in
$F^{n}$ of length $(k+1)$.  By the Solomon--Tits theorem, $\cT$ has
the homotopy type of a bouquet of $(n-2)$-spheres, and in particular
has reduced homology concentrated in dimension $n-2$.  One can
construct classes in $\tilde{H}_{n-2} (\cT)$ by taking the fundamental
classes of apartments: one chooses a basis $E = \{v_{1},\dotsc
,v_{n}\}$ of $F^{n}$ and considers all the possible flags that can be
constructed from $E$ by taking spans of permutations of subsets.  By
appropriately choosing signs one obtains a class $\langle v_{1},\dotsc
,v_{n}\rangle \in \tilde{H}_{n-2} (\cT)$.  It is known that such
classes span the homology.  We have an action of $\bG (\QQ)$, and by
definition, the \emph{Steinberg module} $\St_{n}$ is the $\bG
(\QQ)$-module $\tilde{H}_{n-2} (\cT )$.

According to the Borel--Serre duality theorem \cite{BS}, for any
arithmetic subgroup $\Gamma\subset \bG (\QQ)$ we have
\[
H^{\nu -k} (\Gamma; \CC  ) \longriso H_{k} (\Gamma ; \St_{n}\otimes \CC).
\]
Thus to compute the cohomology of $\Gamma$, we need to take a
resolution of the Steinberg module.  We work with two different
complexes; each has features that help us deal with the two sides of
our computational problem, namely computing $H^{*}$ as well as the
action of the Hecke operators.  The first, the Koecher complex, comes
from the geometry of the Koecher polyhedron.  It has the advantage
that it is finite mod $\Gamma$, but it does not admit an action of the
Hecke operators.  On the other hand the second, the sharbly complex,
does admit a Hecke action, but unfortunately it is not finite mod
$\Gamma$.

Recall that $\Sigma$ is the fan of Koecher cones in the closed cone
$\overline{C}$.  After we mod out by homotheties, the nonzero cones in
$\Sigma$ determine cells in the quotient $\overline{X} =
(\overline{C}\smallsetminus \{0 \})/\RR_{>0}$.  We let $K_{*}$ be the
oriented chain complex on these cells, and let $\partial K_{*}$ be the
subcomplex generated by the cells contained entirely in the boundary
$\partial X = \overline{X}\backslash X$.  The \emph{Koecher complex}
$\cK_{*}$ is the quotient complex $K_{*}/\partial K_{*}$.  We can
construct a map
\begin{equation}\label{eq:koechtostein}
\cK_{1} \longrightarrow \St_{2}
\end{equation}
as follows.  Let $\{e_{1},e_{2} \}$ be the standard basis of
$F^{2}$. By \eqref{2cones} of Proposition \ref{prop:cones}, nontrivial
generators of $\cK_{1}$ correspond to the images of the oriented
$2$-cones in the $\GL_{2} (\OO)$-orbit of the cone spanned by $q
(e_{1}), q (e_{2})$.  Orient this cone by taking the spanning points
in this order, and then map it to the class of the
corresponding apartment.  This leads to a diagram
\[
\cK_{*} \longrightarrow \St_2,
\]
which gives the desired resolution.

We now turn to the sharbly complex.  Recall that for any $x\in
\OO^{2}\smallsetminus \{0 \}$, we have constructed a point $q (x)\in
\overline{C}$ (see \eqref{eq:star}).  Write $x\sim y$ if $q (x)$ and $q (y)$
determine the same point in $\overline{X}$.  Let $\cS_k$, $k\geq 0$, be
the $\Gamma$-module $A_{k}/C_{k}$, where $A_{k}$ is the set of formal
$\CC$-linear sums of symbols $\uu =[x_1, \dots, x_{k+2}]$, where each
$x_i$ is in $\OO ^2\smallsetminus \{0 \}$, and $C_{k}$ is the submodule generated by
\begin{enumerate}
\item $[x_{\sigma(1)}, \dots, x_{\sigma(k+2)}]-\sgn(\sigma)[x_1,
\dots, x_{k+2}]$,
\item $[x, x_2, \cdots, x_{k+2}] - [y, y_2, \dotsc , y_{k+2}]$ if $x\sim
y$, and 
\item $\uu $ if $x_1, \cdots ,
x_{k+2}$ are contained in a hyperplane (we say $\uu $ is \emph{degenerate}).
\end{enumerate}  
We define a boundary map $\partial\colon \cS_{k+1} \to \cS_{k}$ by
\begin{equation}\label{eq:boundary}
\del [x_1, \cdots , x_{k+2}] =\sum_{i=1}^{k+2}(-1)^i[x_1, \cdots, \hat{x}_i,\dotsc  , x_{k+2}],
\end{equation}
where $\hat{x}_{i}$ means omit $x_{i}$.
The resulting complex $\cS_{*}$ is called the
\emph{sharbly complex}.  We have a map $\cS_{0}\rightarrow \St_{2}$
analogous to \eqref{eq:koechtostein}, and the sharbly complex provides
a resolution of $\St_{2}$. 

Recall that all Koecher cones of dimension $\leq 6$ are simplicial
(Proposition \ref{prop:cones}, (i))
Thus one can identify $\cK_{i}$ with a subgroup of $\cS_{i}$ for
$i\leq 5$.  This identification is compatible with the boundary maps,
so we can think of $\cK_{*}$ as being a subcomplex of $\cS_{*}$ in
these degrees.  Since our main focus is computing $H^{4}$, we will
work exclusively with Koecher/sharbly chains in degrees $0,1,2$, and
therefore we regard the corresponding Koecher chains as living in the
sharbly complex.  Propositions \ref{prop:cones}--\ref{prop:stabs} give
all the information we need for computing $H^{4}$.

To compute the Hecke operators, one proceeds as follows.  First one
computes the cohomology group $H^{4} (\Gamma \backslash X)$ using the
relevant part of $\cK_{*}$. Let $\xi$ be a cycle representing a class
in $H^{4}$.  We write $\xi$ as a finitely supported 1-sharbly cycle
$\xi = \sum n(\bu) \bu $, where ``cycle'' means that the boundary
vanishes modulo the action of $\Gamma$.  We can compute the Hecke
operator $T$ in the sharbly complex as 
\begin{equation}\label{eq:heckim}
T (\xi) = \sum n (\bu)\sum_{g} g\cdot\bu,
\end{equation}
where the inner sum is taken over a finite subset of $\bG (\QQ)$.  The
right hand side of \eqref{eq:heckim} will usually be a $1$-sharbly
cycle that does not come from the Koecher complex, and thus cannot
obviously be written in terms of a fixed basis of $H^{4}$.  Therefore
one needs an algorithm to move $T (\xi)$ back to a sum of cycles
coming from $\cK_{*}$.  How this is done is described in the next
section.

\begin{remark}
Since $\cK_{*}$ is not the chain complex of a simplicial complex, one
cannot directly regard it as a subcomplex of $\cS_{*}$.  However, it
is possible to canonically refine $\cK_{*}$ to a complex
$\tilde{\cK}_{*}$ that does sit naturally inside the sharbly complex:
one simply considers the simplicial complex on all simplices that
arise by subdividing the cells in $\cK_{*}$ without adding new
vertices.  One must also add new relations that encode when an
original Koecher cell is a union of simplices.  Since we do not need
this construction in our paper, we omit the details.
\end{remark}

\begin{section}{Reduction algorithm}\label{sec:reduction}

In this section we describe how Hecke images $T (\xi)$ as in
\eqref{eq:heckim} are rewritten as a sum of sharbly cycles supported
on Koecher cones.  The main ideas and steps already appear in
\cite{ants}, which studied the case of a $\GL_{2} (\OO)$, where $\OO$
is the ring of integers of a real quadratic field.  Here we focus on
the differences between that case and the current.  We begin with some
definitions.

Given a point $q (x)$, $x\in \OO^{2}\smallsetminus \{0 \}$, let $R
(x)$ be the unique point $y\in \OO^{2}\smallsetminus \{{0} \}$ such
that $y\sim x$ and $y$ is closest to the origin along the ray
$\RR_{\geq 0} q (x)$.  We call $R (x)$ the \emph{spanning point} of
$x$; for any sharbly $\bu = [x_{1},\dotsc ,x_{n}]$ we can speak of its
set of spanning points.  Any sharbly $\bu = [x_{1},\dotsc ,x_{n}]$
determines a closed cone $\sigma (\bu)$ in $\bar C$.  We call a
sharbly $\bu = [x_{1},\dotsc ,x_{n}]$ \emph{reduced} if its spanning
points are a subset of the spanning points of some fixed Koecher cone;
similarly we call a sharbly cycle reduced if each sharbly in its
support is reduced.  Note that $\bu$ reduced does not mean that
$\sigma (\bu)$ is the face of some perfect pyramid, and thus a reduced
sharbly cycle need not come from a Koecher cycle.  However, it is
clear that there are only finitely many reduced sharbly cycles modulo
$\Gamma$.  Moreover, it is not difficult to write a reduced sharbly
cycle in terms of Koecher cycles directly, so our main challenge is to
rewrite $T (\xi)$ in terms of reduced sharbly cycles.

We now introduce notions of size and reduction level help us gauge how
close a sharbly cycle is to being reduced.  For any $0$-sharbly $\bu =
[x_{1}, x_{2}]$, define the \emph{size} of $\bu $ by $\Size (\bu) =
|\Norm( \det (R (x_{1}), R (x_{2})))|$.  We extend this to general
sharblies by first defining the size of $\bu = [x_{1},\dotsc ,x_{n}]$
to be the maximum of the size of the sub $0$-sharblies
$[x_{i},x_{j}]$, and then to chains $\xi = \sum n(\bu)\bu $ by taking
the maximum size found over the support of $\xi$.

We now focus on $0$- and $1$-sharblies.  Given a $1$-sharbly $\bu =
[x_1,x_2,x_3]$, we call its sub $0$-sharblies \emph{edges}, and define
the \emph{reduction level} of $\bu$ to be the number of edges that are
not reduced.  According to Proposition~\ref{prop:cones} the only
$\GL_2(\OO)$-class of reduced 0-sharbly is $[e_1,e_2]$, and all we can often
detect nonzero reduction level by computing sizes, but we remark that
there are subtleties: for instance, there exist nonreduced 1-sharblies
of reduction level zero.  Similar phenomena occur in \cite{zeta5,
ants}, and reflect the infinite order units in $\OO$.

There are also edges with size 0 that are not reduced.  More
precisely, looking at the edges of the standard 1-sharblies, we see
that $[x, cx]$ is reduced only if the coordinates of $x$ generate the
ideal $\OO$ and $c \in \{\pm t, \pm t^{-1}\} = \{\pm \ep, \pm
\ep^{-1}\}$, where $\ep = -t$ is a generator of $\OO^\times$ mod
torsion.  Such $0$-sharblies are of course degenerate and are
eliminated in the defining relations of $\cS_{*}$, but when one wants
to write a $1$-sharbly cycle in terms of reduced cycles they must be
considered.  Again, this is not surprising since the same phenomena
appear in \cite{zeta5, ants}.  However, if a $1$-sharbly contains an
edge of the $[x, \pm x]$, then the 1-sharbly is degenerate and is
thrown away.

We now turn to the reduction algorithm.  The overall structure
proceeds as described in \cite{ants}, and we refer to there for more
details.  For each $1$-sharbly $\bu$ in the support of a Hecke image,
we $\Gamma$-equivariantly choose a collection of \emph{reducing
points} for the nonreduced edges of $\bu$ (see \cite[\S 3.3]{ants}).
These points, together with the original spanning points of $\bu$, are
assembled into a new $1$-sharbly chain $\xi '$, following the cases
described in \cite[\S 3.5]{ants}.  The process is repeated until the
$1$-sharbly chain is reduced.

In the current work, there are two new features to this algorithm:
\begin{enumerate}
\item \label{item1} The case of reduction level 0 is handled slightly
differently than in \cite{ants,zeta5}.\footnote{We also remark that for reduction level 1, only
the second case of \cite[\S 3.5 (III)]{ants} is needed.}  
\item \label{item2} A new technique is used to choose reducing points $\Gamma$-equivariantly.
\end{enumerate}
For item \eqref{item1} see the next section; item \eqref{item2} is discussed in
\S\ref{sec:normal_forms}.

\section{Reduction level 0}\label{s:rl0} We consider $1$-sharblies
with all edges reduced, but that are not themselves reduced.
According to Proposition~\ref{prop:cones}, there are two $\GL_2(\OO)$-orbits of
2-cones.  One is degenerate, and the other corresponds to a 0-sharbly
of size 1.  It follows that the 1-sharblies with all edges reduced
must have edges of determinant $0$ or $\epsilon^k$, where $\epsilon$
is an infinite generator of the unit group $\OO^\times$, and $k\in \ZZ
$ has large absolute value.  More precisely, write
\[
\OO^\times \simeq \langle -1 \rangle \times
\langle  \epsilon \rangle.
\]  Then up to 
$\GL_2(\OO)$-equivalence, we need only consider $1$-sharblies with
spanning points $e_1,$ $e_2$, and $v = \vect{a\\b}$, where
\begin{itemize}
\item $a =
(-1)^{f_a} \epsilon^{n_a}$ and $b = (-1)^{f_b} \epsilon^{n_b}$,
\item at
least one of $|n_a|$, $|n_b|$ is strictly greater than 1, and 
\item both $n_{a}, n_{b}$ are non-zero.
\end{itemize}

We now describe how to write $[e_{1}, e_{2}, v]$ in terms of reduced
1-sharblies. We do this by replacing $[e_{1}, e_{2}, v]$ by a sum of
1-sharblies that are $\GL_2(\OO)$-equivalent to 3-cones listed in
Proposition~\ref{prop:cones}. 
First introduce new vertices $v_a =
\vect{a\\0}$, $v_b = \vect{0\\b}$ and edges 
$[v_a, v]$,  
$[v_a, e_1]$,  
$[v_a, e_2]$,  
$[v_b, v_a]$, 
$[v_b, v]$, and 
$[v_b, e_2]$
 (cf.~the left of 
Figure~\ref{fig:red}.)

Now we treat each region $A_i$ separately, depending on a variety of cases.
Given $z = (-1)^f \epsilon^n \in \OO^\times$, let $\IInt(a) \subset \OO^\times$
denote the list  
\[\IInt(z) = [z_0, z_1, z_2, \dotsc, z_{|n|}],\]
where $z_i = (-1)^f \epsilon^i$ if $n \geq 0$, and  $z_i = (-1)^f
\epsilon^{-i}$ if $n < 0$.  For example, if $z = -\epsilon^3$, then
$\IInt(z)$ is the list $[-1, -\epsilon, -\epsilon^2, -\epsilon^3]$.

By construction, the region $A_0$ is already reduced.  In fact,
we have 
\[[v, v_a, v_b] = \gamma \cdot [e_1 + e_2, e_1, e_2], \]
where $\gamma =  \mat{a & 0\\0 & b} \in \GL_2(\OO)$. 

For the regions $A_{1}$, $A_{2}$, along the edge $[v_{a}, e_{1}]$ we
introduce vertices 
\[\varepsilon_a = \biggl\{\vect{z\\0}\ \biggm |\ z \in \IInt(a) \biggr\} 
\]
Note that $\#\IInt(a) \geq 2$.  If $\#\IInt(a) = 2$, then no new
vertices are added.  Otherwise, let $\varepsilon_a^0$ be $\varepsilon_a$
with the first and last vector removed.  Construct edges
\begin{gather*}
\{[w,v]\ \mid\ w \in \varepsilon_a^0\}, \quad \{[w,e_2]\ \mid\ w \in
\varepsilon_a^0\}, \quad \text{and} \\
\{[(\varepsilon_a)_i,(\varepsilon_a)_{i+1}]\ \mid \ i = 1, \dotsc,
\#\varepsilon_a - 1\}.
\end{gather*}
In $A_1$ this creates the $1$-sharblies $[v, (\vep_a)_i,
(\vep_a)_{i+1}]$.  These are reduced.  Indeed, if $n_a \geq 0$, then
\[[v, (\vep_a)_i, (\vep_a)_{i+1}] = \gamma \cdot [e_1, e_2, \epsilon
  e_2],\]
where $\gamma = \mat{0 &
    b^{-1}\\ \epsilon^{-i} & ab^{-1} \epsilon^{-i}} \in \GL_2(\OO)$.
If $n_a < 0$, then 
\[[v, (\vep_a)_i, (\vep_a)_{i+1}] = \gamma \cdot [e_1, e_2, \epsilon^{-1}
  e_2],\] where $\gamma = \mat{0 & b^{-1}\\ \epsilon^{i} & ab^{-1}
\epsilon^{i}} \in \GL_2(\OO)$.  A similar argument applies to the new
$1$-sharblies in $A_{2}$, which are $[e_2,(\vep_a)_{i+1}, (\vep_a)_{i}
]$; thus they are also reduced.

The regions $A_{3}$, $A_{4}$ are handled in the same manner.  Along
the edge $[e_2, v_b]$ we introduce the vertices
\[\varepsilon_b = \biggl\{\vect{0\\z}\ \biggm|\ z \in \IInt(b) \biggr\}.\]
As before $\#\IInt(b) \geq 2$, and we put $\varepsilon_b^0$ to be  $\varepsilon_b$
with the first and last vector removed.  Add edges 
\begin{gather*}
\{[w,v]\ \mid \ w \in \varepsilon_b^0\}, \quad \{[w,v_a]\ \mid\ w \in
\varepsilon_b^0\}, \quad \text{and} \\
\{[(\varepsilon_b)_i,(\varepsilon_b)_{i+1}]\ \mid\ i = 1, \dotsc ,
\#\varepsilon_b - 1\}.
\end{gather*}
The 1-sharblies added in region $A_3$ are $[v_a, (\vep_b)_{i},
(\vep_b)_{i+1}]$; those added in region $A_4$ are $[v, (\vep_b)_{i+1},
(\vep_b)_{i}]$.  Both types are reduced as above.  The result, after
completing all steps in this section, is the chain of $1$-sharblies
depicted in the right of Figure~\ref{fig:red}.

\begin{figure}
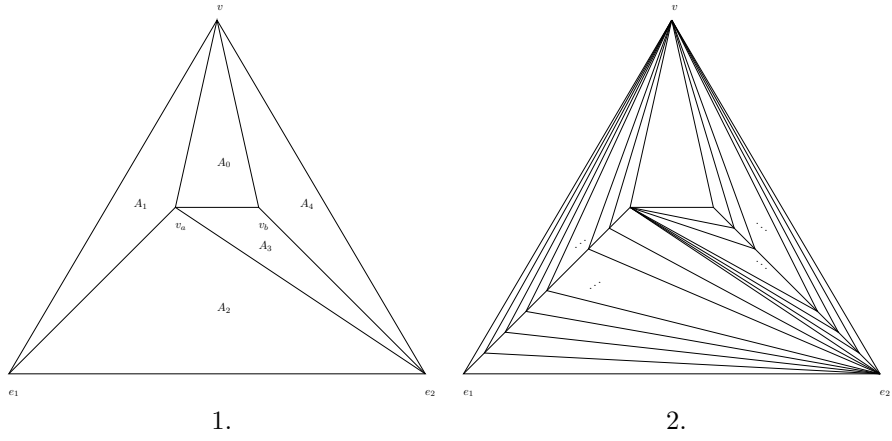

\begin{tabular}{c c}
\includegraphics[width = 0.45\textwidth]{type0_step1.mps} & 
\includegraphics[width = 0.45\textwidth]{type0_step2.mps}\\
1. &2.
\end{tabular}
\caption{Reduction of Type-$0$ $1$-sharbly.} \label{fig:red}
\end{figure}
\end{section}

\begin{section}{Normal forms}\label{sec:normal_forms}

An important aspect of the reduction algorithm is choosing reducing
points $\Gamma$-equivariantly.  This means the following.  Suppose
that $\bu$ and $\bu '$ are two $1$-sharblies that appear in a
$1$-sharbly cycle $\xi$ mod $\Gamma$ with coefficient $1$, and that an
edge $[x_{1},x_{2}]$ of $\bu$ cancels an edge $[x'_{1}, x'_{2}]$ of
$\bu '$ when $\partial \xi$ is taken mod $\Gamma$; that is, there
exists some $g\in \GL_{2}(\OO)$ such that $g\cdot [x_{1}, x_{2}] = -
[x'_{1}, x'_{2}]$.  Suppose further that these two edges are not
reduced.  Then when executing the reduction algorithm we begin by
choosing reducing points $w, w'$ for them, and we require that $g\cdot
w = w'$.  If this $\Gamma$-equivariance is not respected, then the
output of the algorithm will not be a cycle mod $\Gamma$.

In \cite[\S\S 3.2--3.4]{ants}, we presented a technique to do this.
First we encode a $1$-sharbly cycle mod $\Gamma$ as a collection of
$1$-sharblies with extra data attached to their edges, called
\emph{lift matrices}.  These are $2\times 2$ matrices over $\OO$ such
that if two edges are $\GL_{2} (\OO)$-equivalent, then their lift
matrices are in the same $\GL_{2} (\OO)$-orbit.  Then when computing a
reducing point for an edge, we first compute a unique representative
of the $\GL_{2} (\OO)$-orbit of its lift matrix.  This 
representative is then fed into the reducing point algorithm.

Thus we must explain how to compute a normal form for any lift matrix.
We use the technique of pseudo-matrices and their associated Hermite
normal form \cite{cohen-ntbook}.

\begin{defn}
  A \emph{pseudo-matrix} is a pair $(I, A)$, where $A = (a_{j,l})$ is an $m \times
  r$ matrix with entries in $F$, and $I = (\fa_i)$ is a sequence of
  fractional ideals called \emph{coefficient ideals}.
\end{defn}

\begin{defn}
  A pseudo-matrix $(I,A)$ is in \emph{Hermite normal form} if there
  are $s \leq m$, $1 \leq i_1 < \dots < i_s$ such that 
  for $1 \leq l < i_j$, we have $a_{j,l} = 0$, $a_{j,i_j} = 1$, and
  $a_{j,l}$ is reduced  modulo $\fa_j \fa_l^{-1}$, and $a_{j,l} = 0$
  for $j > s$.  
\end{defn}

We specialize to our case $m = k = 2$.  For a matrix $A$, let $A_i$
denote the $i${th} row.  For a matrix $A \in \Mat_2(\OO)$, let $M_A$
denote the $\OO$-module generated by $\{A_1,A_2\}$.  A pseudo-matrix
consists of a triple $[I_1,I_2,A]$, where $I_i \subseteq \OO$ are
ideals, and $A \in \Mat_2(\OO)$.  For a pseudo-matrix $\cA =
[I_1,I_2,A]$, let $M_\cA$ denote the $\OO$-module
\[M_\cA = \{a_1A_1 + a_2A_2 \mid a_i \in I_i\}.\]
Thus if $A \in \Mat_2(\OO)$ and $\cA = [\OO,\OO,A]$, then $M_A =
M_\cA$.

Given a pseudo-matrix $\cA = [I_1,I_2,A]$, Magma's \verb+HermiteForm+
function returns a pair $\mathcal{H}, T$, where 
$\mathcal{H}$ is pseudo-matrix $[J_1,J_2,H]$ in Hermite normal form, and $T
\in \GL_2(F)$ satisfies
\[M_\cH = M_\cA \quad \text{and} \quad H = TA.\]

Note that while $A$ and $H$ have entries in $\OO$, the entries of $T$
lie in $F$ in general.  Note, however, that $J_i\cdot T_i \subseteq
\OO^2$ for $i = 1,2$.  Since $F$ is class number one, we can pick
generators $g_i \in \OO$ such that $J_i = g_i \OO$.  Then $g_i T_i \in
\OO^2$, and thus $\tilde{T} = \mat{g_1 T_1\\g_2 T_2} \in \GL_2(\OO)$.
Let $\tilde{H} = \mat{g_1 H_1\\g_2 H_2}$.  If $\cA$ comes from a
matrix $A$ so that $\cA = [\OO,\OO,A]$, then we have $\tilde{H} =
\tilde{T}A$.

If we wish to use the theory of Hermite normal forms to produce normal
forms for matrices using the remark above, we need a section to the
map
\[\OO \to \OO/\OO^\times.\]
Let $a \in \OO$ be the element for which we wish to produce a
representative modulo units.  First construct the ideal $\fa =
a\OO$. Then compute a $\ZZ$-basis for $\fa$, and compute the Hermite
normal form (over $\ZZ$) of the matrix of this basis.  This will
produce a sequence $L \subset \OO$ which generates $\fa$.   In Magma,
we create another ideal $\fb$ generated by 
$L$.  Note that 
\begin{enumerate}
\item $\fb = \fa$.
\item The generators $\fb$ only depend on the ideal $\fa$, not the
  original choice of generator $a$.
\end{enumerate}
Since $F$ is class number one, Magma will produce a generator $a_0$
for $\fb$, which only depends on the generators $L$.  This $a_0$ is
the representative for $a$ modulo units.

This procedure, combined with Hermite normal form for pseudo matrices
and the remark above allow us to, given $A \in \Mat_2(\OO)$, output a
Hermite normal form $\tilde{H}$ and transformation matrix $\tilde{T}$
such that 
\begin{enumerate}
\item $\tilde{H} = \tilde{T}A$.
\item If $A' = gA$ for some $g \in \GL_2(\OO)$, then $\tilde{H'} = \tilde{H}$.
\end{enumerate}

This completes the computation of normal forms, but there is still one
implementation detail of interest.  Recall that we really only need
the normal form to ensure all reducing points are chosen
$\Gamma$-equivariantly.  A key step in computing a reducing point is
computing the cone containing the barycenter of the line through
$q(v_1)$ and $q(v_2)$, where $v_i$ are the columns of the lift matrix
(cf.~\cite[\S 3.3]{ants}).  By abuse of notation, we will call this
the cone containing the lift matrix.

It turns out that for a typical lift matrix $A$ arising in our
computations, the cone containing $A$ is much closer to the
``standard" cones (i.e., cones appearing as faces of the perfect
pyramids from Theorem \ref{thm:perfect}) than the cone containing the
Hermite normal form $\tilde{H}$.  This means that if we use Hermite
normal form as the input to our code that computes reducing points,
then our code will waste considerable time trying to compute the cone
containing $\tilde{H}$.

We address this issue as follows by caching normal forms as we compute
them.  Given a matrix $A \in \Mat_2(\OO)$, we compute the pseudo
matrix Hermite normal form $\cH$.  If $\cH$ shows up on our master
list, we return the associated normal form $A_0$.  If $\cH$ is not on
the master list, we set $A_0 = A$ to be the normal form associated to
$\cH$.  In practice we go ahead and compute the reducing point for $A$
as well before adding the data to the master list.
\end{section}

\section{Computational results}\label{s:cr}

We conclude by presenting our computational results, from both the
cohomology and elliptic curve sides.  The techniques are very similar
to that of \cite{zeta5}, so we will be brief. As before our programs
were implemented in Magma \cite{magma}.  We remark that, as in
\cite{zeta5}, to avoid floating-point precision problems we did not
use the complex numbers $\CC$ as coefficients for $H^{4}$, but instead
computed cohomology with coefficients in the large finite field
$\FF_{\ourprime}$.  We expect that the Betti numbers we report coincide
with those one would compute for the group cohomology with
$\CC$-coefficients.

As in \cite{zeta5}, we began by computing the dimension of $H^{4}$
for a large range of level.  To determine if $H^{4}_{\cusp} \not = 0$,
we experimentally determined the dimensions of the
subspace $H^{4}_{\Eis}$ spanned by \emph{Eisenstein cohomology
classes} \cite{harder.gl2}.  Such classes are closely related to
Eisenstein series.  In particular the eigenvalue of $T_{\fq}$ on these
classes equals $\Norm (\fq) +1$.  We expect that for a given level
$\fn$, the dimension of the Eisenstein cohomology space depends only on
the factorization type of $\fn$.  Thus initially we used some Hecke
operators applied to cohomology spaces of small level norm to compute
the expected Eisenstein dimension for small levels with different
factorization types.  The result can be found in
Table~\ref{tab:eisenstein}.  After compiling this table, we were able
to predict which levels had cohomology in excess of the Eisenstein
subspace, and thus gave candidates for cuspidal classes.

Next we computed the Hecke operators. Our primary focus was to find
eigenclasses with rational eigenvalues, and in fact this was almost
always the case: for all levels except level norm $529 = 23^{2}$, the
Hecke eigenvalues were rational.  Altogether we computed the
cohomology at $308$ different levels, which includes all ideals with
level norm $\leq 835$.

We now turn to the elliptic curve side.  We compiled a list of
elliptic curves over $F$ of small norm conductor simply by searching
over a box of Weierstrass equations.  More precisely, for a positive
integer $B$, let
$$S_B = \bigr\{ c_0 + c_1 t +c_2
t^2 \bigm| |c_i| \leq B, 0\leq i \leq 2 \bigr\},$$ be a boxed
grid of size $(2B+1)^3$ inside the lattice of algebraic integers in $F$,
centered at the origin.  
We considered
equations of the form 
\[
E: y^2 +
a_1 xy + a_3 = x^3 + a_2 x^{2} + a_4 x + a_6, \text{with $a_1, a_2, a_3, a_4, a_6 \in S_{B}$}.
\]
and took $B = 2$.  For each non-zero discriminant $\leq 10000000$, we
kept those of norm conductor $\leq 20000$ to arrive at a list of
elliptic curves of small norm conductor.  This consisted of 26445
curves lying in 1518 isomorphism classes.  Of course with such a
naive method we have no way of knowing whether or not we have found
all isomorphism---or even isogeny---classes of curves up to some
bound.  Nevertheless, this enables us to obtain a list of curves of
various conductors, which must then be reconciled with the cohomology
data.

Finally we compare the two sides.  In every case, we found perfect
agreement:
\begin{enumerate}
\item In all cases where an eigenspace was one-dimensional with
rational Hecke eigenvalues, we found an elliptic curve over $F$ whose point counts over
finite fields agreed with our Hecke data, at least within the range
where we were able to compute both sides.  There were $44$ levels
where this occurred; the curves are given in Table \ref{tab:ec}.
Generators for the conductors of these curves are given in Table
\ref{tab:cond}. 
\item For no level/conductor of norm $\leq 835$ did we find a curve
over $F$ that was not accounted for by a Hecke eigenclass, or a Hecke
eigenclass that could not be matched to a curve.
\item In $10$ levels we encountered a two-dimensional eigenspace on
which the Hecke operators acted by rational scalar matrices.  These
were ``old'' cohomology classes and can be accounted for by cohomology
classes at lower levels.  Table \ref{tab:old} records the data.
\item At norm level $529 = 23^{2}$, we found a two-dimensional
cuspidal subspace where the eigenvalues live in $\QQ (\sqrt{5})$.
The eigenvalues of this cohomology class match those of the weight two newform
of level $23$.  

\end{enumerate}

\bibliographystyle{amsplain_initials} 
\bibliography{complexcubic}    

\providecommand{\bysame}{\leavevmode\hbox to3em{\hrulefill}\thinspace}
\providecommand{\MR}{\relax\ifhmode\unskip\space\fi MR }
\providecommand{\MRhref}[2]{%
  \href{http://www.ams.org/mathscinet-getitem?mr=#1}{#2}
}
\providecommand{\href}[2]{#2}
\begin{thebibliography}{10}

\bibitem{AGM}
A.~Ash, P.~E. Gunnells, and M.~McConnell, \emph{Cohomology of congruence
  subgroups of {${\rm SL}\sb 4(\mathbb Z)$}}, J. Number Theory \textbf{94}
  (2002), no.~1, 181--212.

\bibitem{AGM5}
A.~Ash, P.~E. Gunnells, and M.~McConnell, \emph{Resolutions of the {S}teinberg
  module for {$GL(n)$}}, J. Algebra \textbf{349} (2012), 380--390.

\bibitem{amrt}
A.~Ash, D.~Mumford, M.~Rapoport, and Y.-S. Tai, \emph{Smooth compactifications
  of locally symmetric varieties}, second ed., Cambridge Mathematical Library,
  Cambridge University Press, Cambridge, 2010, With the collaboration of Peter
  Scholze.

\bibitem{Aru}
A.~Ash and L.~Rudolph, \emph{The modular symbol and continued fractions in
  higher dimensions}, Invent. Math. \textbf{55} (1979), no.~3, 241--250.

\bibitem{BS}
A.~Borel and J.-P. Serre, \emph{Corners and arithmetic groups}, Comment. Math.
  Helv. \textbf{48} (1973), 436--491, Avec un appendice: Arrondissement des
  vari\'et\'es \`a coins, par A. Douady et L. H\'erault.

\bibitem{magma}
W.~Bosma, J.~Cannon, and C.~Playoust, \emph{The {M}agma algebra system. {I}.
  {T}he user language}, J. Symbolic Comput. \textbf{24} (1997), no.~3-4,
  235--265, Computational algebra and number theory (London, 1993).

\bibitem{bygott}
J.~Bygott, \emph{Modular forms and modular symbols over imaginary quadratic
  fields}, Ph.D. thesis, Exeter, 1999.

\bibitem{cohen-ntbook}
H.~Cohen, \emph{Advanced topics in computational number theory}, Graduate Texts
  in Mathematics, vol. 193, Springer-Verlag, New York, 2000.

\bibitem{cremona2}
J.~E. Cremona, \emph{Hyperbolic tessellations, modular symbols, and elliptic
  curves over complex quadratic fields}, Compositio Math. \textbf{51} (1984),
  no.~3, 275--324.

\bibitem{crem.whitley}
J.~E. Cremona and E.~Whitley, \emph{Periods of cusp forms and elliptic curves
  over imaginary quadratic fields}, Math. Comp. \textbf{62} (1994), no.~205,
  407--429.

\bibitem{dembele}
L.~Demb{\'e}l{\'e}, \emph{Explicit computations of {H}ilbert modular forms on
  {${\Bbb Q}(\sqrt{5})$}}, Experiment. Math. \textbf{14} (2005), no.~4,
  457--466.

\bibitem{dsv}
M.~Dutour~Sikiri{\'c}, A.~Sch{\"u}rmann, and F.~Vallentin, \emph{A
  generalization of {V}oronoi's reduction theory and its application}, Duke
  Math. J. \textbf{142} (2008), no.~1, 127--164.

\bibitem{Fra}
J.~Franke, \emph{Harmonic analysis in weighted {$L\sb 2$}-spaces}, Ann. Sci.
  \'Ecole Norm. Sup. (4) \textbf{31} (1998), no.~2, 181--279.

\bibitem{fulton}
W.~Fulton, \emph{Introduction to toric varieties}, Annals of Mathematics
  Studies, vol. 131, Princeton University Press, Princeton, NJ, 1993, The
  William H. Roever Lectures in Geometry.

\bibitem{zeta5}
P.~E. Gunnells, F.~Hajir, and D.~Yasaki, \emph{Modular forms and elliptic
  curves over the field of fifth roots of unity}, Experiment. Math. (to
  appear), 2011.

\bibitem{Gmod}
P.~E. Gunnells, \emph{Modular symbols for {${\bf Q}$}-rank one groups and
  {V}orono\u\i\ reduction}, J. Number Theory \textbf{75} (1999), no.~2,
  198--219.

\bibitem{Gu}
P.~E. Gunnells, \emph{Computing {H}ecke eigenvalues below the cohomological
  dimension}, Experiment. Math. \textbf{9} (2000), no.~3, 351--367.

\bibitem{ants}
P.~E. Gunnells and D.~Yasaki, \emph{Hecke operators and {H}ilbert modular
  forms}, Algorithmic number theory, Lecture Notes in Comput. Sci., vol. 5011,
  Springer, Berlin, 2008, pp.~387--401.

\bibitem{harder.gl2}
G.~Harder, \emph{Eisenstein cohomology of arithmetic groups. {T}he case {${\rm
  GL}_2$}}, Invent. Math. \textbf{89} (1987), no.~1, 37--118.

\bibitem{koecher}
M.~Koecher, \emph{Beitr\"age zu einer {R}eduktionstheorie in
  {P}ositivit\"atsbereichen. {I}}, Math. Ann. \textbf{141} (1960), 384--432.

\bibitem{lingham}
M.~Lingham, \emph{Modular forms and elliptic curves over imaginary quadratic
  fields}, Ph.D. thesis, Nottingham, 2005.

\bibitem{Man}
J.~I. Manin, \emph{Parabolic points and zeta functions of modular curves}, Izv.
  Akad. Nauk SSSR Ser. Mat. \textbf{36} (1972), 19--66.

\bibitem{sw}
R.~G. Swan, \emph{Generators and relations for certain special linear groups},
  Advances in Math. \textbf{6} (1971), 1--77 (1971).

\bibitem{voronoi1}
G.~Vorono\v\i, \emph{Sur quelques propri\'et\'es des formes quadratiques
  positives parfaites}, J. Reine Angew. Math. \textbf{133} (1908), 97--178.

\end{thebibliography}

\newpage

\begin{table}[htb]
\begin{center}
  \begin{tabular}{|c|c|c|c|c|c|c|}
\hline
Factorization of $\fn$ & $\fp$ & $\fp^2$ & $\fp\fq$ & $\fp^3$ & $
\fp\fq^2$ & $\fp \fq \fr$\\
\hline
$\dim H^4_{\Eis}(\Gamma_0(\fn))$ &\fatten{3} & \fatten{5} & \fatten{7} & \fatten{7} & \fatten{11} & \fatten{15}\\
\hline
  \end{tabular}
\medskip
\caption{Expected dimension of Eisenstein cohomology
  $H^4_{\Eis}(\Gamma_0(\fn))$ in terms of the prime factorization of
  $\fn$. Prime ideals are denoted by $\fp, \fq, \fr$.\label{tab:eisenstein}}
\end{center}
\end{table}

\begin{table}[htb]
\begin{tabular}{|c|ccccc|}
\hline
$\Norm(\fn)$&$a_1$&$a_2$&$a_3$&$a_4$&$a_6$\cr
\hline
89 & $t-1$&$-t^2-1$&$t^2-t$&$t^2$&$0$\cr
107 & $0$&$-t$&$-t-1$&$-t^2-t$&$0$\cr
115 & $-t^2+t-1$&$-t^2+1$&$t-1$&$-1$&$-t^2$\cr
136 & $-t^2$&$-1$&$-t^2+1$&$t+1$&$0$\cr
161 & $t^2-t-1$&$-t^2+t-1$&$t^2-t+1$&$t^2-t$&$t-1$\cr
167 & $t^2+1$&$t+1$&$t^2+t-1$&$-t^2-t+1$&$-t^2+t+1$\cr
185 & $t$&$-t^2+t+1$&$t+1$&$0$&$0$\cr
223 & $1$&$t^2$&$t^2+t-1$&$-t^2+t-1$&$1$\cr
253 & $-1$&$-t^2-t$&$-t^2-t$&$-t^2-t$&$0$\cr
259 & $0$&$1$&$-t^2-t-1$&$t^2-t+1$&$-t^2-t+1$\cr
275 & $-t^2+t$&$t$&$t^2-t$&$0$&$0$\cr
289 & $-1$&$t^2-t$&$t$&$1$&$0$\cr
293 & $t^2-1$&$-t+1$&$t^2-t+1$&$0$&$0$\cr
344 & $t-1$&$-t^2-t$&$-t^2+t+1$&$t^2-1$&$0$\cr
359 & $-t^2+1$&$t+1$&$t^2+1$&$t^2-t$&$-t+1$\cr
385 & $-t^2$&$-t^2-t-1$&$-t^2-1$&$t^2+t$&$-t^2+1$\cr
392 & $-t^2+1$&$-t^2+t+1$&$-t+1$&$t^2-1$&$t$\cr
440 & $-t^2+1$&$-t^2-t+1$&$-t^2$&$-t$&$0$\cr
449 & $-t^2$&$1$&$-t^2+t+1$&$t+1$&$0$\cr
475 & $0$&$-t$&$t^2+t$&$t^2-t+1$&$-t^2+1$\cr
503 & $-t^2+t$&$-t^2+t+1$&$-t^2+t$&$-t^2+t$&$0$\cr
505 & $t^2-t$&$t^2-t+1$&$t^2+t$&$t^2-t+1$&$-t^2+1$\cr
505 & $t^2-t-1$&$t^2-1$&$0$&$t-1$&$0$\cr
512 & $0$&$t^2+1$&$0$&$t^2$&$0$\cr
553 & $t$&$1$&$t$&$0$&$0$\cr
593 & $t$&$-t-1$&$t^2$&$-t^2+t+1$&$0$\cr
595 & $-t^2+t-1$&$-t^2+t+1$&$-t^2+t+1$&$-t^2+t+1$&$0$\cr
625 & $t$&$-t-1$&$t^2+1$&$1$&$-t^2$\cr
649 & $-t^2-t-1$&$-t$&$0$&$-t^2+t-1$&$0$\cr
665 & $-t$&$-t^2+1$&$-t^2+t$&$-t^2+t$&$0$\cr
685 & $t^2-1$&$-t^2+t$&$-t^2+1$&$-t-1$&$t^2$\cr
712 & $2t^2-t-1$&$-t^2-2t+2$&$t+2$&$2t^2+2t$&$-2t^2-t$\cr
719 & $-t^2+t$&$-t^2+t-1$&$-1$&$0$&$0$\cr
719 & $t^2-t-1$&$-t^2+t-1$&$0$&$t^2-t$&$0$\cr
721 & $-t^2+t+1$&$-t^2+t-1$&$t^2+1$&$-t-1$&$-t+1$\cr
727 & $1$&$t^2+t-1$&$t^2-t$&$-1$&$0$\cr
773 & $2t^2-1$&$2t+1$&$2t^2+2t$&$-t^2+2t+1$&$-2t^2-t$\cr
805 & $-t^2-2t+2$&$-2t^2+2t$&$t^2-t-1$&$-2t^2+t+2$&$-2t^2-t$\cr
808 & $-t-2$&$0$&$2t^2-t-2$&$-2t^2+2t+2$&$-2t^2-t$\cr
809 & $-t^2+t-1$&$t^2-1$&$t^2+1$&$t^2-t$&$-t^2$\cr
809 & $t^2-1$&$t^2+t-1$&$t^2-t$&$t^2-t$&$0$\cr
817 & $-t^2+t$&$-t^2$&$t^2-t+1$&$-1$&$0$\cr
829 & $0$&$-t^2$&$t^2-t-1$&$0$&$0$\cr
\hline
\end{tabular}
\medskip
\caption{Equations for elliptic curves over $F$\label{tab:ec}. Here $t$ is a root of $x^3-x^2+1$.}
\end{table}

\begin{table}[htb]
\begin{center}
\begin{tabular}{|l|c|l|c|l|c|l|c|}
\hline
N($\fn$) & generator&N($\fn$) & generator&N($\fn$) & generator&N($\fn$) & generator\cr
\hline
89 & $4t^2 - t - 5$& 107 & $-5t^2 + 3t$& 115 & $-2t^2 - 2t - 3$& 136 & $6t^2 - 2t - 2$\cr
161 & $-5t^2 + 5t + 4$& 167 & $-5t^2 + 3t - 3$& 185 & $-t^2 - 5t + 4$& 223 & $6t^2 - 5t - 2$\cr
253 & $7t^2 - 5t - 5$& 259 & $4t^2 - 7t - 1$& 275 & $8t^2 - 2t - 3$& 289 & $3t^2 - 7t - 2$\cr
293 & $-5t^2 - 2t - 2$& 344 & $6t^2 - 2t - 8$& 359 & $7t^2 - 6t - 2$& 385 & $-6t^2 + 7t + 5$\cr
392 & $-8t^2 + 6t + 6$& 440 & $8t^2 + 2t - 6$& 449 & $t^2 - 8t$& 475 & $-4t^2 - 7t$\cr
503 & $t^2 - t - 8$& 505 & $-2t^2 - 7t + 2$& 505 & $-8t + 1$& 512 & $8$\cr
553 & $9t^2 - 4t - 2$& 593 & $8t^2 - t - 9$& 595 & $11t^2 - 4t - 6$& 625 & $8t^2 + 3t + 1$\cr
649 & $8t^2 + t - 8$& 665 & $9t^2 + t - 8$& 685 & $-7t^2 + 5t - 7$& 712 & $6t^2 - 10t - 8$\cr
719 & $t^2 - t - 9$& 719 & $11t^2 - 4t - 5$& 721 & $8t^2 - 9$& 727 & $10t^2 - 7t - 7$\cr
773 & $-3t^2 + 12t - 5$& 805 & $3t^2 - 6t - 10$& 808 & $-6t^2 - 2t - 4$& 809 & $9t^2 - 9t - 1$\cr
817 & $-t^2 - 7t - 8$& 829 & $6t^2 - t - 10$& &&&\cr
\hline
\end{tabular}
\end{center}
\medskip
    \caption{Generators for ideals arising as conductors of elliptic curves.\label{tab:cond} Here $t$ is a root of $x^3-x^2+1$.  The two curves labelled 809 in Table \ref{tab:ec} share the level of norm 809 in this table.  The curves labelled 505 and 719 have different conductors, as indicated by the levels in this table.}
\end{table}

\begin{table}[htb]
\begin{center}
\begin{tabular}{|l|c|c|}
\hline
N($\fn$) & generator& Norm of original level\cr
\hline
445 & $2t^2 - 5t - 8$  & 89 \cr
535 & $-7t^2 + 8t + 2$  & 107\cr
575 & $-9t^2 + t + 9$  & 115\cr
623 & $10t^2 + 2t - 7$  & 89\cr
680 & $-6t - 8$   & 136\cr
712 & $6t^2 - 10t - 8$  & 89\cr
749 & $-8t^2 + t - 2$  & 107\cr
805 & $-t^2 + 10t + 4$  &  161\cr
805 & $3t^2 - 6t - 10$  & 115\cr
835 & $-10t^2 + 8t - 1$  & 167\cr
\hline
\end{tabular}
\end{center}
\medskip \caption{``Old'' cohomology classes.  In every instance the
eigenspace was two-dimensional, with the Hecke operators acting by
scalar matrices.\label{tab:old} The eigenvalues originally occur at
the levels in the third column.  Here $t$ is a root of $x^3-x^2+1$.}
\end{table}

\end{document}